\documentclass[12pt]{article}

\usepackage{fancyhdr, amsmath, amssymb, amsfonts, url, subfigure, dsfont, mathrsfs, graphicx, subfigure, amsthm, authblk}
 \usepackage{tikz}

\newlength{\guillotine}
\newtheorem{thm}{Theorem}[section]
\newtheorem{cor}[thm]{Corollary}
\newtheorem{lemma}[thm]{Lemma}

\newtheorem{prop}[thm]{Proposition}

\newtheorem{definition}[thm]{Definition}

\theoremstyle{remark}
\newtheorem{rem}[thm]{Remark}
\begin{document}
\date{}

\title{Comparison theorems for closed geodesics \\  on negatively curved surfaces}
\author{Stephen Cantrell} 
\author{Mark Pollicott\thanks{Supported by the ERC Grant 833802-Resonances and  EPSRC grant EP/T001674/1. }}
\affil{Warwick University}

\newcommand{\Addresses}{{
  \bigskip
  \footnotesize

  S.~Cantrell, \textsc{Department of Mathematics, Warwick University,
    Coventry, CV4 7AL-UK}\par\nopagebreak
  \textit{E-mail address} \texttt{s.j.cantrell@warwick.ac.uk}

\medskip
  M.~Pollicott, \textsc{Department of Mathematics, Warwick University,
    Coventry, CV4 7AL-UK}\par\nopagebreak
  \textit{E-mail address} \texttt{masdbl@warwick.ac.uk}
}}

\maketitle

\begin{abstract}
In this note we present new asymptotic  estimates  comparing the word length and geodesic length of closed geodesics on surfaces with (variable) negative sectional curvatures. 
In particular, we provide an averaged comparison of these 
two important quantities and obtain precise statistical results, including a central limit theorem and a local limit theorem. 
Further, as a corollary we also improve an asymptotic formula of  R. Sharp  and the second author \cite{PS-average}. 
Finally, we
relate our results to recent work of Gekhtman, Taylor and Tiozzo \cite{gtt}.
\end{abstract}

\maketitle

\section{Introduction}
\subsection{background}
Let $V$ be a compact $C^\infty$ Riemannian surface with negative Gaussian curvature.
As is well known, each non-trivial  conjugacy class in the fundamental group $\pi_1(V)$ contains exactly one closed geodesic, thus we see there is a countable infinity of  closed geodesics on $V$. 

Let $\gamma$ denote a closed geodesic on $V$.
We can then associate to $\gamma$ two natural 
weights:
\begin{enumerate}
\item The length $l(\gamma)$ of the closed geodesic, i.e., its length in the Riemannian metric on $V$;
\item   The word length $|\gamma|$ of the closed geodesic, i.e., the smallest number of a fixed family of generators in the standard presentation needed to represent any element in the conjugacy class in $\pi_1(V)$
corresponding to the curve $\gamma$.
\end{enumerate}
A continuous change in the Riemannian metric leads to a change in the lengths of the closed geodesics but not the word length.
However, a change in the choice of  generators does lead to a change in the word length.

It is  well known that the lengths of closed geodesics tend to infinity and there is a  famous  result due to Margulis from 1969 \cite{margulis1} on their asymptotics
(see also \cite{margulis2}). More precisely, let $\pi(T)$ be the number of closed geodesics of length at most $T$ then 
$$
\pi(T) \sim \frac{e^{hT}}{hT}, \hbox{ as } T \to \infty
$$
where $h > 0$ denotes the topological entropy.
A more refined version with an error term is that
$$
\pi(T) = \hbox{\rm li}(e^{hT})
+ O\left( e^{(h-\epsilon) T} \right)
\hbox{ where }
 \hbox{\rm li}(x) := \int_2^{x} \frac{du}{\log u}. \eqno(1.1)
$$
Let us now consider a standard presentation
$$
\langle
a_1, \cdots, a_g, b_1, \cdots, b_g \hbox{ : }
\prod_{i=1}^g[a_i, b_i] = I
\rangle.
$$
where $g \geq 2$ is the genus of the surface.
Let $\Gamma = \{a_1^{\pm 1}, \cdots, a_g^{\pm 1}, b_1^{\pm 1}, \cdots, b_g^{\pm 1}\}$ be the symmetric family of generators.
We can define the word length of a closed geodesic 
$\gamma$ by
$$
|\gamma| = \inf \{k \hbox{ : } [\gamma] = [g_1 \cdots g_k] \hbox{ with }  g_1,  \cdots g_k \in \Gamma \}
$$
where $[\cdot ]$ denotes the conjugacy class.
\footnote{Other  presentations can be used but for simplicity we will concentrate on this one}

Another classical  result, a consequence of a deeper result of  Milnor,  is  that
the word length and geometric length are closely related \cite{milnor}. In particular, there exists a constant $C > 1$  such that for any closed geodesic $\gamma$  one has that 
$$
\frac{1}{C} \leq \frac{|\gamma|}{l(\gamma)} \leq C. \eqno(1.2)
$$

By taking a more statistical viewpoint one can get  finer  detail on the distribution of the values 
$\frac{|\gamma|}{l(\gamma)}$.
We begin by recalling a basic  asymptotic expression which we obtain when we average over 
closed geodesics $\gamma$ ordered by their lengths $l(\gamma)$.

\begin{prop}[\cite{PS-average}]
There exists $A > 0$  such that 
$$
\lim_{T \to \infty}  
\frac{1}{\pi(T)}
\sum_{l(\gamma) < T}
\frac{|\gamma|}{l(\gamma)} = A. \eqno(1.3)
$$
\end{prop}
 In the next section we present an asymptotic formula which significantly
 sharpens this result.  We shall then further improve upon this result by proving   
the corresponding Central Limit Theorem and Local Central Limit Theorem.

\subsection{New comparison results}
In this subsection we present a stronger version of the asymptotic formula $(1.3)$ in which we have an error term.  This will be a consequence of the following asymptotic average over word lengths. 

\begin{thm} \label{thm}[Asymptotic average of word length] \label{asymav}
There exists constants $A \in \mathbb{R}$ and  $\beta>0$ such that
$$ \frac{1}{\pi(T)} \sum_{l(\gamma) < T} |\gamma| = \frac{A}{h}\frac{e^{hT}}{\textnormal{li}(e^{hT})} \left( 1 + O(e^{-\beta T}) \right) 
\hbox{ as } T \to \infty.\eqno(1.4)
$$
\end{thm}
This result was partly motivated by a result of Cesaratto and Vall\'ee on quadratic numbers \cite{vallee}.
The expression in (1.4) has the advantage of having an exponential error term.  However, using the standard expansion for $\hbox{\rm li}(\cdot)$ we could write it in a more elementary, or less informative way, as
$$
 \frac{1}{\pi(T)} \sum_{l(\gamma) < T} |\gamma| =AT  + A_0 +  \sum_{k=1}^N \frac{A_n}{T^n} + O\left( \frac{1}{T^{N+1}}\right) 
\hbox{ as } T \to \infty 
$$
where $N \in \mathbb N$ and $A_0, A_1, A_2, \cdots $ are constants (with, for example, $A_0 = \frac{A}{h}$).

As an immediate application of Theorem \ref{thm}, we deduce the following improvement on the 
asymptotic in (1.2).

\begin{cor}\label{secondcor}
There exists $\beta>0$ such that
$$\frac{1}{\pi(T)} \sum_{l(\gamma)<T} \frac{|\gamma|}{l(\gamma)} = A + O(e^{-\beta T}).$$
\end{cor}

We will prove this corollary after the proof of Theorem \ref{asymav}. \\ \indent

\subsection{Central Limit Theorems}
Another finer  form of statistical result, comparing the word length and geodesic length of individual closed geodesics, are  Central Limit Theorems.
We begin with another asymptotic expression which can be used to introduce the value of a variance $\sigma^2$.

\begin{thm}[Asymptotic Variance]\label{asymvar}
There exists $\sigma^2>0$, $D \in \mathbb R$
 and $\beta>0$ and 
$$\frac{1}{\pi(T)} \sum_{l(\gamma)<T} \left(|\gamma| - \frac{1}{\pi(T)} \sum_{l(\gamma)<T} |\gamma| \right)^2 = \sigma^2T  + D + O(T^{-1}) \hbox{ as $T \to\infty$.}$$
\end{thm}
We then prove the following Central Limit Theorem which compares the word length and geometric length of individual closed geodesics. Our methods show that the Berry-Esseen error term holds.

\begin{thm}[Central Limit Theorem] \label{clt}
	For any real values $a < b$ we have 
$$ \frac{1}{\pi(T)} \# \left\{ l(\gamma) < T : a \leq \frac{|\gamma| - A T}{\sqrt{T}} \le b \right\} = \frac{1}{\sqrt{2\pi} \sigma} \int_{a}^b e^{-u^2/2\sigma} \ dt + O\left( \frac{1}{\sqrt{T}} \right) \hbox{  as $T \to \infty$}.$$
\end{thm}

In \cite{gtt} Gekhtman, Taylor and Tiozzo proved a different type of Central Limit Theorem where closed geodesics are ordered by their word length, 
motivated as they were by the experimental results of Chas, Li and Maskit \cite{CLM}. 
However, the ordering by length of closed geodesics   in  Theorem \ref{clt} seems more in keeping with the counting results of Margulis \cite{margulis1} 
and others.

There is also a local central limit theorem with an error term.  This is a more subtle asymptotic estimate on the difference of the individual terms 
 $l(\gamma)$ and $|\gamma|$, but without the scaling by $\sqrt{T}$.

 \begin{thm}[Local Central Limit Theorem] \label{llt}	For $x \in \mathbb R$ we have
 	$$
 	\frac{1}{\pi(T)}
 	\# \left\{ 
 	  l(\gamma) < T \hbox{ : }
 	x -\frac{1}{2}  < 
 	|\gamma| - A  T
 	\leq  x +\frac{1}{2}
 	\right\}
 	= \frac{  e^{-x^2/2\sigma T} }{\sqrt{2\pi}\sigma \sqrt{T}} 
 	+ O\left( \frac{1}{\sqrt{T}} \right),
 	$$
 	as $T\to\infty$.
 \end{thm}

To obtain Theorem \ref{asymvar} and Theorem \ref{clt}, we combine our analysis with the results of Hwang from  \cite{hwang1} and \cite{hwang2} (see also \cite{fs} for a nice account). Theorem \ref{asymav} and Theorem \ref{llt} will follow from more classical arguments.\\ \indent

 Of particular interest is the dynamical interpretation and properties of the (new)  
 constants  $D$ and $\sigma^2$.

\begin{prop}\label{constants}
	The constants $A, \sigma^2$ and $D$  have a smooth dependence on the Riemannian metric on the surface.
\end{prop}

Each can be explicitly written in terms of the geodesic flow, via the associated zeta function and its derivatives (see Section $5.1$).
This leads to the conclusion of the above proposition.

We now briefly discuss some results related to this work.
In \cite{lalley} Lalley  proved a 
Local Limit Theorem, from which can be derived  a Central Limit Theorem,
comparing $l(\gamma)$ and 
$l_f(\gamma) = \int_\gamma f \ d\mu$, 
where $f : V \to \mathbb{R}$ is a H\"older function 
satisfiying a non-lattice condition.
and $\mu$ is the projection onto the surface 
of the
measure of maximal entropy for the geodesic flow.  
Recently,  Sharp and the first author provided a direct proof of a more general  Central Limit Theorem 
\cite{CS}. 

\indent

In sections 2-4, we present the proof of Theorem \ref{thm}.  In sections 5-6, we present the proofs 
Theorems \ref{asymvar}, \ref{clt} and \ref{llt}.
In Section $7$,  we compare Theorem \ref{clt} with the corresponding Central Limit Theorems considered in \cite{gtt}.
In particular, we established a   stronger Berry-Esseen error term.

To end this introduction we briefly describe the methods we use to prove our results. To prove Theorem \ref{thm} we study the domain of analyticity of a generating function of one variable. To do this we use tools from ergodic theory and more specifically, thermodynamic formalism. We then translate this `analytic information' into our results using techniques from 
analytic
number theory.  Theorem \ref{thm} is in some sense the 
starting point for the subsequent  more delicate  statistical results. To obtain 
these
theorems we need to work with a more complicated function that allows us to keep track of more information than is possible using
 a function of a single complex variable.
More specifically, we define and study a zeta function of two variables. Carrying out the analysis for this new function is much more difficult than analysing its one variable counterpart but, of course, this leads to significantly stronger results. 
In particular,  we
need to develop new methods that allow us to deal with the additional problems that arise due to the introduction of a second variable. 
%
Sections 5 and 6 are devoted to this  end.

In the next section we  begin by  introducing the essential tools in our analysis: complex functions and symbolic dynamics.

\section{A complex generating  function}

To prove Theorem \ref{asymav} we introduce the following complex  function.
 We can formally define
 	$$\eta_V(s) := \sum_{n=1}^\infty
	n
 	\sum_{|\gamma|=n}
 	e^{ - sl(\gamma)},\quad  s\in \mathbb C$$
where  the  sum over $\gamma$ is taken over all (not necessarily prime) 
directed closed geodesics.
We will show that $\eta_V(s) $ converges to an analytic function for $Re(s)>h$ and has a meromorphic extension to a larger domain.  Our approach is to associate
a Markov interval map to the geodesic flow.

  
  \subsection{Geodesics and Markov maps}
We begin by recalling some basic constructions. We can associate to our surface $V$ and its geodesic flow  $\phi_t:SV \to SV$ on the unit tangent bundle $SV$ the following:
\begin{enumerate}
\item a Markov piecewise $C^1$ expanding map $T: I \to I$ of the interval, and
\item a piecewise $C^1$ function $r: I \to \mathbb{R}_{>0}$.
\end{enumerate}

More precisely, we can replace $I$  by $\coprod_{i=1}^kI_i$, a disjoint union of intervals, and then for each $1\leq i \leq k$ we assume 
$T: I_i \to I$ is $C^1$  with $\inf_{x\in I_i} |T'(x)| > 1$ and the image $T(I_i)$ is a union of some of these intervals.  We similarly assume that $r: I_i \to \mathbb R$ is $C^1$, for $1 \leq i \leq k$.  

Furthermore, we can associate to the map $T: I\to I$ a subshift of finite type $\Sigma_A^+$ with symbols $\{1,\ldots, k\}$ and aperiodic zero-one transition matrix $A$:
$$\Sigma_A^+ = \{ (x_n)_{n=0}^\infty : x_n \in \{1,\ldots,k\}, A_{x_k,x_{k+1}}=1\}, \hspace{2mm} \text{ and } \hspace{2mm} \sigma(x_n)_{n=0}^\infty = (x_{n+1})_{n=0}^\infty.$$
There is a H\"older function $\tilde{r}: \Sigma_A^+ \to \mathbb{R}_{\ge 0} $ that models $r: I \to \mathbb{R}_{\ge 0}$ in this setting. Here we mean H\"older with respect to the $d_\theta$ metric (for some $0< \theta <1$): for $x= (x_n)_{n=0}^\infty,y = (y_n)_{n=0}^\infty\in \Sigma_A^+$, $d_\theta(x,y) = \theta^n,$ where $n$ is the largest integer such that $x_k = y_k$ for all $k \le n$. \\
\indent These systems provide us with a setting in which we can rewrite $\eta_V$ in terms of dynamical quantities. The following result allows us to bridge the gap between the geometry and these dynamical systems.  

  \begin{lemma} \label{codings}
 Given a compact surface $V$ of negative curvature there exists a $C^1$ Markov interval map $T$ and a $C^1$ function $r$ as above, such that
  there is a one-one correspondence between:
  	\begin{enumerate}
  		\item closed geodesics $\gamma$ of length $l(\gamma)$ and word length $|\gamma|$; and
  		\item closed orbits $\{x, T x, \cdots, T^{n-1}x\}$ in $I$ with $l(\gamma) = r^n(x) := \sum_{j=0}^{n-1} r(T^j x)$ and $|\gamma| = n$,
  	\end{enumerate}	
  	 with a finite number of exceptions. Furthermore, there is a one-one correspondence between:
  	 \begin{enumerate}
  		\item closed orbits $\{x, T x, \cdots, T^{n-1}x\}$ in $I$ with $l(\gamma) = r^n(x) := \sum_{j=0}^{n-1} r(T^j x)$ and $|\gamma| = n$; and
  		\item closed orbits $\{x, \sigma x, \cdots , \sigma^nx\}$ in $\Sigma_A^+$ with $l(\gamma) = \tilde{r}^n(x)$ and $|\gamma| = n$.
  	\end{enumerate}
  \end{lemma}
  
  \begin{proof}
  The geodesic flow $\phi_t:SV \to SV$ is well known to be an Anosov flow.
  The coding of closed orbits for the Anosov geodesic flow on the three dimensional manifold corresponding to the unit tangent bundle dates back to Ratner \cite{ratner} and Bowen  \cite{bowen}.  The subshift of finite type models the (Markov) Poincar\'e return map on a finite number of local two dimensional sections and the roof function reflects the transition time between sections. 
  Moreover, because the boundaries of these sections is one dimensional this implies that there is a one-to-one correspondence between the closed orbits for the geodesic flow and the subshift of finite type, up to a finite number of exceptions.  
  However, the orginal Ratner-Bowen construction of the subshift doesn't necessarily reflect the word length of closed geodesics.  To accommodate this we require the more geometric coding of Bowen-Series
  \cite{bowenseries},  \cite{series} for the representation of the (constant curvature) geodesic flow on $(\partial \mathbb H^2 \times \partial \mathbb H^2 - \Delta)/\Gamma$ where $\partial \mathbb H^2$ is the boundary of hyperbolic space, $\Delta$ is the diagonal, and $\Gamma$ is the associated Fuchsian group.  Although this construction is performed in the setting of constant curvature surfaces, it carries over to variable curvature by using that the space of metrics is connected and taking the new sections to be the images of the original sections under the associated homeomorphism which is a conjugacy, up to flow equivalence.
  \end{proof}

 It will sometimes be more convenient to work with the derived
 symbolic model $\Sigma_A^+$  than the interval map $T:I\to I$. Given Lemma \ref{codings} we  can define
$$\eta(s) =  \sum_{n=1}^\infty 
 	 \sum_{\sigma^n x =x}
  e^{ - s\tilde{r}^n(x)}
  $$
 where the inner summation is over periodic points $\sigma^n x=  x \in \Sigma_A^+$.
Then $\eta_V(s) = \eta(s) + \Psi(s)$, where $\Psi(s)$ is an analytic  function on 
 the half-plane
 $Re(s) > 0$ correcting the contributions from the finite number of exceptions described in Lemma \ref{codings}.
   Hence the series  $\eta(s)$ converges provided 
 $$
  e^{P(-\mathrm{Re}(s)\tilde{r})} := \limsup_{n \to \infty}
 \left(\sum_{\sigma^nx=x} \exp( - \mathrm{Re}(s) \tilde{r}^n(x))\right)^{\frac{1}{n}} < 1
 $$

\noindent and the value $P(\cdot)$ has a useful alternative expression thanks to the variational principle: For a (H\"older) continuous function $g: \Sigma_A^+ \to \mathbb R$  we can associate the pressure
 $$
 P(g) = \sup 
 \left\{h(\mu) + \int g d\mu \hbox{ : } \mu = \sigma-\hbox{invariant probability}\right\}.
 $$
 Moreover, there is a unique $\mu_g$  such that $P(g) = h(\mu_g) + \int g d\mu_g$ called the {\it equilibrium state} for $g$.  In particular, the Bowen-Margulis measure of maximal entropy corresponds to the equilibrium state $\mu_{-h \tilde{r}}$, by a classical result of Bowen and Ruelle \cite{bowenruelle}.

In order to prove Theorem \ref{asymav} we need the   following proposition which describes the domain of analyticity of $\eta(s)$.

\begin{prop} \label{thmdom1}
There there exists $\epsilon > 0$ and $C\in \mathbb{R}$
  such that 
$$
\eta(s) = \frac{C}{s- h} + \chi_0(s)
$$
where 
\begin{enumerate}
	\item
 the function $\chi_0(s)$  is analytic on 
 $\mathrm{Re}(s)  > h-  \epsilon$  and 
\item there exists $\xi, M>0$ such that for $|\mathrm{Im}(s)|\ge 1$ we have $\left|\eta(s)\right|\le M|\mathrm{Im}(s)|^\xi$.
\end{enumerate}
Furthermore, by choosing $\epsilon>0$ sufficiently small we can take  $0< \xi <1$.
\end{prop}

The last comment is a simple consequence of the Phragm\'en Lingdel\"of principle.

\noindent The following section is dedicated to proving this result.

\section{Proof of Proposition 2.2 }

The approach to proving Proposition \ref{thmdom1} depends on the use of transfer operators.  First in the context of subshifts of finite type
and then in the setting of expanding maps,
 and in particular the results of Dolgopyat \cite{dolgopyat} that provide bounds on the spectral radii of the operators in the latter setting. 

\subsection{Transfer operators for shift spaces}
	
We recall the definitions of transfer operators for subshifts of finite type.
Let $F_\theta$ denote the collection of real valued H\"older continuous functions (with respect to $d_\theta$) on $\Sigma_A^+$. 

\begin{definition}
We can  define a family of  transfer operators $\widetilde{\mathcal L}_s : F_\theta \to F_\theta$ by
$$\widetilde{\mathcal L}_s w(x) = \sum_{\sigma y=x} e^{-s\tilde{r}(y)}w(y)$$
where $w\in F_\theta$.
\end{definition}
The following result is well known.
\begin{lemma}[Ruelle]\label{ruelle}
	For $\sigma \in \mathbb R$, $ \widetilde {\mathcal L}_\sigma $ has a real simple eigenvalue of maximum modulus which we denote by $\lambda(\sigma) = e^{P(-\sigma \tilde{r})}$ and a corresponding positive eigenfunction $f_\sigma$
such that $\widetilde{\mathcal L}_\sigma f_\sigma = e^{P(-\sigma \tilde{r})} f_\sigma$.
\end{lemma}
Provided $s$ is real and close to $h$, let $Q(s)$ denote the one  projection onto the eigenspace corresponding to the simple maximal eigenvalue of $\widetilde{\mathcal L}_{s}$ acting on $C^1(I)$. We can then write
$$\widetilde{\mathcal L}_{s} = e^{P(-s\tilde{r})} Q(s) + P(s)$$
where $P(s) = I- Q(s)$.
We can further define the linear operators $\widetilde{\mathcal L}_{s,z} = e^z \widetilde{\mathcal L}_s$, i.e., by multiplying by $e^z$ where $z \in \mathbb C$. By the Ruelle Operator Theorem, 
 the operator $\widetilde {\mathcal L}_s$ has a simple maximal eigenvalue $e^{P(-s\tilde{r})}$ and uniform spectral gap to the rest of the spectrum.
By analytic perturbation theory we have the following.

\begin{lemma}[Perturbation theory] \label{pert}
	Provided $\delta>0$ is sufficiently small and $|s-h| < \delta$  
	then $\widetilde {\mathcal L}_{s}$ has a simple eigenvalue which we denote by $e^{P(-s\tilde{r})}$
	and $s \mapsto e^{P(-s\tilde{r})}$ varies analytically.
	There is also an associated  eigenfunction
	$f_s$  such that 
	$\widetilde {\mathcal L}_{s} f_{s} = e^{P(-s\tilde{r})}f_{s}$
	 and $s \mapsto f_{s}$ is analytic.
	 
	Therefore, provided $\delta>0$ is sufficiently small and $|s-h| < \delta$  
	then $\widetilde {\mathcal L}_{s,z}$ has the simple eigenvalue $e^{z}e^{P(-s\tilde{r})}$.
		There is also an associated  eigenfunction
		$f_{s,z} = e^zf_s$ such that 
		$\widetilde{\mathcal L}_{s,z} f_{s,z} = e^{P(z-s\tilde{r})}f_{s,z}$
		(where $f_{s,z} = f_{s}$)
		 and $s \mapsto f_{s,z}$ is therefore analytic.
\end{lemma}
As a consequence, for $n \geq 1$, we can write
$$\widetilde{\mathcal L}_{s,z}^n = e^{nz}e^{nP(-s\tilde{r})} Q(s) + e^{nz}P(s).$$
where $Q$ and $P$ are analytic for $z$ in a neighbourhood of $0$ and $s$ in a neighbourhood of $h$.
All of these perturbation results have analogous statements for the $\widetilde{\mathcal L}_s$ operators.

\subsection{For $s$ close to $h$}


To study $\eta(s)$ for $s$ close to $h$ we begin with the following result.

\begin{lemma} [Theorem $5.5$ \cite{PP2}]
There exists $\epsilon, \delta >0$ and $0<\theta <1$ such that for $|z|< \delta$ and $|s-h|<\epsilon,$ 
$$e^{nz}  \sum_{\sigma^nx=x} e^{-s\tilde{r}^n(x)} = e^{nP(z-s\tilde{r})} + O(\theta^n),$$
uniformly in $s,z$.
\end{lemma}

The proof of  Proposition \ref{thmdom1} depends on treating two different regions for $s$ differently. The first region  is where $s$ is close to $h$ .
The second is where  $s$ is far from $h$. We will study the behaviour of $\eta$ in these regions and then collate our results to prove Proposition $\ref{thmdom1}$.

To study the behaviour of $\eta(s)$ when $s$ is close to $h$, it is easier to take the symbolic view point and so in this section, we work with $\Sigma_A^+$ and $\tilde{r} : \Sigma_A^+ \to \mathbb{R}$.

\begin{lemma} 
	There exists $\alpha, \epsilon > 0$ so that 
	we can write 
	$$
\eta(s) =  \frac{ e^{P(-s\tilde{r})}}{1 - e^{P(-s\tilde{r})}} + \Psi_0(s)
	$$
	where $\Psi_0(s)$ is analytic for  $\mathrm{Re}(s) > h - \epsilon$, $|\mathrm{Im}(s)| \leq \alpha$. 
\end{lemma}

\begin{proof}
We sketch a proof. By Lemma $3.4$, when $\textnormal{Re}(P(-s\tilde{r})) <0$, $\eta(s)$ is well defined, and
$$\eta(s) = 
 \sum_{n=1}^\infty e^{nP(-s\tilde{r})} + {\Psi}_0(s) $$
for ${\Psi}_0(s)$ that is analytic in  $Re(s) > h - \epsilon$.  Both sides of the above expression are analytic and non-zero when $\textnormal{Re}(P(-s\tilde{r})) <0$. Hence we can
write
$$ \eta(s) =  
\frac{e^{P(-s\tilde{r})}}{1-e^{P(-s\tilde{r})}} + \Psi_0(s). $$
 This provides the required analytic extension.
\end{proof}

It then follows easily that $\eta$ has a simple pole of residue $C$ (for some $C \in \mathbb{R}$) at $s=h$.

\subsection{Transfer operators for expanding interval maps}

The transfer operator for expanding interval maps is analogous to that for subshifts of finite type.

\medskip
\begin{definition}
Let $\mathcal L_{s}: C^1(I) \to C^1(I)$ be the family of transfer operators defined by
$$
\mathcal L_{s} w(x) = \sum_{Ty=x} e^{-sr(y)} w(y),
\quad \hbox{ for $w \in C^1(I)$},
$$
for $s \in \mathbb C$.
\end{definition}

All of the perturbation results in Lemma \ref{pert}  have analogous statements for the ${\mathcal L}_s$ operators.

\indent We now turn to the study of $\eta(s)$ for $s$ away from $h$. In this setting it is
better to use these  transfer operators
$\mathcal L_s$
 (instead of the previous operator $\widetilde{\mathcal L}_s$ for the  symbolic model) to exploit important features of the flow (i.e., non-joint integrability and the $C^1$ properties of the flow).  Fortunately, we are  able to adapt the necessary parts of \cite{PS-error} without too much difficulty.

 Given $|t| \geq 1$ we can consider the Banach space  $C^1(I)$
 with the norm
 $$
 \|h\|_t := \max \left\{
 \|h\|_\infty, \frac{\|h'\|}{|t|}
 \right\},
 $$   
and define transfer operators $\mathcal{L}_s : C^1(I) \to C^1(I)$ analogously to before. 
We then have the following key result.

\begin{lemma}[Dolgopyat \cite{dolgopyat}] \label{dolgo}
	There exists $\epsilon > 0$, $C > 0$ and $0 < \rho < 1$ such that whenever 
	$|t| \geq 1$ and $\mathrm{Re}(s) > h - \epsilon$ then 
	when $n = m [\log |t|] + l$, 
	where $0 \leq l \leq [\log |t|]-1$ and $m \geq 0$, then 
	$$
	\| \mathcal L_s^n\|  \le C \rho^{m [\log |t|]}e^{ l P(-\sigma r)}
	$$ 
	where $e^{P(\sigma)}$ is the spectral radius of $\mathcal L_s$ and $s = \sigma + it$.  
\end{lemma}

In particular, we see that for  $|t| \geq 1$ and $\mathrm{Re}(s) > h - \epsilon$
the operator $\mathcal L_s$ has spectral radius at most $0<\rho < 1$. 

\subsection{The region $|\mathrm{Im}(s)| \geq 1$}
We require the following result 
(corresponding to Lemma 2 in \cite{PS-error}).   Let $I = \coprod_{i=1}^k I_i$ be the union of disjoint intervals on which $T,r$ are $C^1$. Write $\chi_i : I \to \mathbb{R}$ for the indicator function on $I_i$.

\begin{lemma} \label{ruellecomp}
	For each $I_i$, we can take $x_i \in I_i$ such that there exists $C>0$ and  $0 < \rho_0<1$ independent of $I_i$ with
	$$
	\left|\sum_{T^nx=x} e^{-sr^n(x)} - \sum_{i=1}^k \mathcal L_{s}^n \chi_{i}(x_i)\right|
	\leq C |t|  n \rho_0^n.
	$$
\end{lemma}
\begin{proof}
See Lemma $2$ of \cite{PS-error}.
\end{proof}
 An immediate consequence of this and Lemma \ref{dolgo} is that $\eta(s)$ is analytic and non-zero in the domain $\mathrm{Re}(s)>h-\epsilon, |\mathrm{Im}(s)|\ge 1$.  We now recall that, by the weak-mixing properties of the geodesic flow, $\eta(s)$ cannot have any poles on the line $\mathrm{Re}(s)=h$ \cite{PP2}. Using a simple compactness argument it follows that, after possibly reducing $\epsilon$, 
 $\eta(s)$ is analytic and non-zero on $\mathrm{Re}(s)>h-\epsilon$ and $|\mathrm{Im}(s)|\ge \alpha$. 
 This domain of analyticity passes to the logarithmic derivative.\\ \indent

Now, using Lemma \ref{ruellecomp} we can bound, for some $C'>0$,
$$
\begin{aligned}
\left|\sum_{T^nx=x} e^{-sr^n(x)}\right|
& \leq 
	\left|\sum_{T^nx=x} e^{-sr^n(x)} - \sum_{i=1}^k \mathcal L_{s}^n \chi_{I_i}(x_i)\right| + \left|\sum_{i=1}^k \mathcal L_{s}^n \chi_{i}(x_i)\right| \cr
	&\leq C |t|  n \rho_0^n + \sum_{i=1}^k \|\mathcal L_{s}^n\|\cr
	&\leq  C |t|  n \rho_0^n + k C' \rho^{m [\log |t|]} e^{ l P(- \sigma r)}.
\end{aligned}\eqno(3.1)
$$
Taking $h'$ with $h-\epsilon < h' < h$ , we consider $\eta(s)$ for
$\mathrm{Re}(s) > h'$. By $(3.1)$, there exists constants $C_1, C_2 >0$ independent of $z,s$ such that
$$
\begin{aligned}
\left|\eta(s) \right| &\leq C_1 \sum_{n=1}^\infty 
 |t|   \rho_0^n
+ kC_2 \sum_{m=0}^\infty 
\rho^{m[\log |t|]}
\left(\sum_{l=0}^{[\log |t|]-1}e^{lP(-\sigma r)}\right) \cr
&\leq \frac{C_1 \rho_0}{1-\rho_0 }|t|
+ kC_2 \left(
\frac{1}{1 - |t|^{-|\log \rho|}}
\right) \left( \frac{e^{[\log|t|] } e^{[\log |t|] P(- h' r)}-1}{e^{P(- h' r)} -1} \right)
\end{aligned}
$$
This implies the required bound for Proposition \ref{thmdom1}.

\section{Proof of Theorem 1.2}
We begin the proof of Theorem \ref{asymav} by introducing some classical techniques  used in  analytic number theory.

\subsection{The Perron Formula}
We use the following standard result from complex analysis.
\begin{lemma} [The uniform Effective Perron Formula (\cite{tenenbaum}, Section $\textnormal{II}.2$, Theorem $2$)]\label{perron}
Define
$$  \theta(y) = \left\{
     \begin{array}{@{}l@{\thinspace}l}
       0&\hspace{4pt} 0<y<1\\
       1/2&\hspace{4pt} y=1 \\
       1&\hspace{4pt} y>1.\\
     \end{array}
   \right.$$
Then for any fixed $ \xi >0$ there exists $K>0$ such that,
$$\left|\frac{1}{2\pi i} \int_{d-iR}^{d+iR} \frac{y^s}{s} ds - \theta(y)\right| \le \frac{K y^d}{1+R|\log(y)|},$$
 uniformly for $y>0$, $R\ge 1$, $d<\xi$.
 \end{lemma}
The statement of Perron's Formula in \cite{tenenbaum} does not have uniformity in the variable $d <\xi$, however this follows easily from an inspection of the proof.
From Lemma \ref{perron}
and Proposition  \ref{thmdom1}
we can  deduce the following lemma. The notation $\sum'_{l(\gamma') \le T}$ is used to express the normal summation except that it weights elements corresponding to $l(\gamma')=T$ by a half.

\begin{lemma} \label{qp1}
There exists  $\alpha >0$ such that
$$\sideset{}{'}\sum_{l(\gamma') \le T}
|\gamma'| = \frac{C}{h} e^{hT}  \left( 1 + O\left(e^{-\alpha T} \right) \right)$$
where $C$ is as in Proposition \ref{thmdom1} as $T\to\infty$.
\end{lemma}

\begin{proof}
Set $d = h + T^{-1}$, $\epsilon =e^{-xT}$ and $R= e^{yT}$  (for $x,y >0$ to be chosen later). Using Perron's Formula gives that
$$\sideset{}{'}\sum_{l(\gamma') \le T} |\gamma'| = \frac{1}{2\pi i} \int_{d-iR}^{d+iR} - \eta(s) \frac{e^{sT}}{s}  ds + \sum_{\gamma'} O\left( \frac{ |\gamma'|e^{d(T-l(\gamma'))}}{1+R|T-l(\gamma')|}  \right).$$
As noted in (1.1),  we have $\pi(T) = \textnormal{li}(e^{hT}) + O(e^{h'T})$ for some $h' < h$  \cite{PS-error}, 
and so if  $\alpha \in \mathbb{R}$ satisfies $0<\alpha<h-h'$, then
$$\#\{ \gamma': |l(\gamma') - T| \le e^{-\alpha T}\} = O(e^{(h-\alpha)T}).$$
Hence, if $0<x<h-h'$ and $|T-l(\gamma')|<\epsilon$, then $e^{d(T-l(\gamma'))} = O(1)$ and so
$$\sum_{ |T-l(\gamma')|\le \epsilon}   \frac{ |\gamma'|  e^{d(T-l(\gamma'))}}{1+R|T-l(\gamma')|} = O\left(T  \epsilon e^{hT} \right),$$
where we have used from (1.2) that $|\gamma'|< C l(\gamma')$ for some $C>0$.
On the other hand,  we also have 
$$ \sum_{ |T-l(\gamma')|> \epsilon}   \frac{ |\gamma'| 
 e^{d(T-l(\gamma'))}}{1+R|T-l(\gamma')|} \le \frac{e^{dT}}{R\epsilon} \sum_{\gamma'} |\gamma'|
 e^{-dl(\gamma')} = \frac{e^{dT}}{R\epsilon} (\eta(d) +O(1)).$$
Using the properties of $\eta(s)$ we have $\eta(d) = O(T)$ and so the above is
$O\left( \frac{ T e^{dT}}{R\epsilon}\right).$

Now fix $c<h$ such that $\eta$ is analytic in $\mathrm{Re}(s) > c$ apart from the pole at $s=h$. We integrate around the rectangle $\Gamma$ with vertices $d-iR, c-iR, c+iR,d+iR$. By the Residue Theorem 
$$\frac{1}{2\pi i} \int_\Gamma -\eta(s) \frac{e^{sT}}{s} ds = \frac{C}{h}e^{hT}.$$
Our bounds on  $\eta(s)$ guarantee that
$$\int_{c \pm iR}^{d \pm iR} -\eta(s) \frac{e^{sT}}{s} ds = O(R^{\xi-1} e^{dT})$$
and
$$\int_{c-iR}^{c+iR} -\eta(s) \frac{e^{sT}}{s} ds = O(R^\xi e^{cT}).$$
We now choose $y$ such that $R^\xi e^{cT}$ grows strictly exponentially slower that $e^{T}$.
We then  choose $0 < x < h-h'$ so that $ T e^{dT}/R\epsilon$ grows strictly exponentially slower than $e^{hT}$.
\end{proof}

A standard argument shows that the same expression from Lemma  \ref{qp1} holds when we sum over prime periodic orbits. Furthermore, recalling our observation that for small $x >0$ and $\epsilon=e^{-x T}$
$$\#\{\gamma: |l(\gamma) - T|<\epsilon\} =  O(\epsilon e^{hT}),$$
it is easy to remove the terms corresponding to $l(\gamma) = T$ from our asymptotic expression. 
Finally, we deduce that there exists  some $\alpha>0$ for which

\begin{equation*}
 \sum_{l(\gamma ) < T}    |\gamma|  =   \frac{C}{h} e^{hT} \left( 1 + O\left( e^{-\alpha T} \right) \right). \eqno(4.1)
\end{equation*}

\subsection{Proofs  of Theorem $1.2$ and Corollary \ref{secondcor} }
We are now ready to complete the proof of  Theorem  \ref{asymav} 

\begin{proof}[Proof of Theorem $1.2$]
As stated in (1.1) we have  that $\pi(T) = \textnormal{li}(e^{hT}) + O(e^{h'T})$ for some $h'<h$ \cite{PS-error}. Dividing (4.1) by this expression gives that
$$\frac{1}{\pi(T)} \sum_{l(\gamma)<T} |\gamma| = \frac{A}{h} \frac{ e^{hT}}{\textnormal{li}(e^{hT})} ( 1 + O(e^{-\beta' T}))$$
(the identification of $C$ with $A$ being apparent from (1.3))
for some $\beta'>0$ as required.
\end{proof}
We can now deduce Corollary \ref{secondcor} from this result. 

\begin{proof}[Proof of Corollary 1.3]
The following expression will be useful:
\begin{align*}
\int_2^T \frac{e^{hu}}{u^2} \ du 
&=h \ \text{li}(e^{hT}) - \frac{e^{hT}}{T} + O(1).
\end{align*}
This follows from the substitution $ u = \log x /h$
and integration by parts.
We now define 
$$\varphi(u)= \sum_{l(\gamma)<u} |\gamma| \eqno(4.2)$$ and note that
$$\sum_{l(\gamma)<T} \frac{|\gamma|}{l(\gamma)} = \int_2^T \frac{1}{u} \ d\varphi(u) + O(1) = \frac{\varphi(T)}{T} + \int_2^T \frac{\varphi(u)}{u^2} du +O(1).$$
Using our expression (4.2) for $\varphi(u)$  we deduce that there exists $h' <h$ such that
\begin{align*}
\sum_{l(\gamma)<T} \frac{|\gamma|}{l(\gamma)} &= \frac{A}{h} \frac{e^{hT}}{T} + \frac{A}{h} \int_2^T \frac{e^{hu}}{u^2} du + O(e^{h'T})\\
&=\frac{A}{h}\frac{e^{hT}}{T} + \frac{A}{h} \left( h \ \text{li}(e^{hT}) - \frac{e^{hT}}{T}\right) + O(e^{h'T})\\
&= A \ \text{li}(e^{hT}) + O(e^{h'T}).
\end{align*}
Dividing through by $\pi(T)$ then concludes the proof.
\end{proof}


\section{Zeta functions}
Although we were able to prove the results thus far using functions of a single variable, to proceed with the remaining results we need to consider functions of two variables.
A well known approach to the classical theorem of Margulis
(used in \cite{PP1}, see also \cite{PP2},  and inspired by the proof of the Prime Number Theorem,  rather than in the original work of Margulis \cite{margulis1}) is to use a zeta function of one variable of the form
$$
\zeta(s) 
=\prod_{\gamma \in \mathcal{P}}
\left(
1 - e^{- sl(\gamma)}
\right)^{-1}, \quad s \in \mathbb R,
$$
which converges to a non-zero analytic function for $\mathrm{Re}(s) > h$, where $h$ is the topological entropy of the associated geodesic flow.
The classical approach to studying this zeta function is to use symbolic dynamics.   It is known in this case that there exists $\epsilon > 0$, 
so that we can write 
 $$
 \zeta(s) = - \frac{1}{s- h} + \chi_0(s)
 $$
 where 
 \begin{enumerate}
 	\item
 	the function $\chi_0(s)$  is analytic and non-zero on 
 	$\mathrm{Re}(s) > h-\epsilon$ ; and 
 	\item there exists $\xi>0$ such that for $|\mathrm{Im}(s)|\ge 1$ we have $|\zeta(s)|\le C |\mathrm{Im}(s)|^\xi$.
\end{enumerate}
In particular, $\zeta(s)$ has a simple pole at $s=h$.

 \subsection{A zeta function in two variables}
 We can formally define a zeta function in two variables by
 $$
 \zeta_V(s,z) = \prod_\gamma \left( 1 - z^{|\gamma|} e^{-sl(\gamma)}\right)^{-1}
 $$
 where $s,z\in \mathbb C$.
Using this as a model we can now formulated a dynamical zeta function associated to our shift space.

 \begin{definition}
 	We can formally define
 	$$\zeta(s,z) =\exp  \left( \sum_{n=1}^\infty 
 	\frac{e^{nz}}{n} \sum_{T^n x =x}
  e^{ - sr^n(x)}\right),\quad  s, z \in \mathbb C$$
 	which converges to an analytic function for $\mathrm{Re}(s)$ sufficiently large and $\mathrm{Re}(z)$ sufficiently small. The sum over $x$ is taken over periodic points for $T: I \to I$.
 \end{definition}
 
 We can write $\zeta(s,z)  = \zeta_V(s,z)  + \Psi(z,s)$ where $\Psi(z,s)$ is analytic and non-zero for $\mathrm{Re}(s)>0$ and $z$ close to $1$.
 
As before, it will sometimes be more convenient to work with the symbolic model $\Sigma_A^+$  than our interval map $T:I\to I$. 
Given Lemma \ref{codings}, we can write
$$\zeta(s,z) =  \exp  \left( \sum_{n=1}^\infty 
 	\frac{e^{nz}}{n} \sum_{\sigma^n x =x}
  e^{ - s\tilde{r}^n(x)}\right)$$
where the sum over $x$ is taken over periodic points for $\sigma: \Sigma_A^+\to \Sigma_A^+$.
 Hence the series in $\zeta(s,z)$ converges provided 
 $$
 |e^{z}| <  e^{-P(-\mathrm{Re}(s)\tilde{r})} = \limsup_{n \to \infty}
 \left(\sum_{\sigma^nx=x} \exp( - \mathrm{Re}(s) \tilde{r}^n(x))\right)^{-\frac{1}{n}}.
 $$

 In the definition of $\zeta(s,z)$ the  $s$ variable  keeps track of the geodesic lengths and the $z$ variable keeps track of the word length. 
In this respect,  $\zeta(s,z)$ has advantage over  the usual single variable zeta function $\zeta(s)$ (just for lengths of closed geodesics) or  even the  generating function (just for word lengths of free homotopy classes) \cite{epstein}, \cite{ghys}.

\subsection{Properties of $\zeta(s,z)$ and its logarithmic derivative}
All of our remaining results are consequences  of the  following proposition which describes the domain of analyticity of the logarithmic derivative
$\frac{\zeta'(s,z)}{\zeta(s,z)},$ where we differentiate with respect to $s$.

\begin{prop} \label{thmdom2}
There there exists $\epsilon, \delta>0$ and an  analytic function
 $$\sigma: \{z \in \mathbb C : |z| < \delta\} \to \mathbb{C} $$
  such that:
 for each  $|z| < \delta$ we can write 
$$
 \frac{\zeta'(s,z)}{\zeta(s,z)} = - \frac{1}{s- \sigma(z)} + \chi_0(s,z)
$$
where 
\begin{enumerate}
	\item
 the function $\chi_0(s,z)$  is bi-analytic on 
 $|\mathrm{Re}(s) - h| < \epsilon$ and $|z|< \delta$; and 
\item there exists $\xi, M>0$ such that for $|\mathrm{Im}(s)|\ge 1$ we have $\left|\frac{\zeta'(s,z)}{\zeta(s,z)}\right|\le M|\mathrm{Im}(s)|^\xi$.
\end{enumerate}
Furthermore, we can choose $\epsilon, \delta$ so that $0< \xi <1$.
\end{prop}

To study the behaviour of $\zeta(s,z)$ when $s$ is close to $h$, it is easier to take the symbolic view point and so in the following lemma, we work with $\Sigma_A^+$ and $\tilde{r} : \Sigma_A^+ \to \mathbb{R}$.

\begin{lemma} 
	There exists $\alpha, \epsilon, \delta > 0$ so that 
	we can write 
	$$
	\frac{\zeta'(s,z)}{\zeta(s,z)} = - \frac{P'(-s\tilde{r}) e^{z+P(-s\tilde{r})}}{1 - e^{z + P(-s\tilde{r})}} + \Psi_0(s,z)
	$$
	where $\Psi_0(s,z)$ is bi-analytic for $|z| < \delta$ and $|s-h|< \epsilon$, $|\mathrm{Im}(s)| \leq \alpha$. 
\end{lemma}

\begin{proof}
We sketch a proof referring the reader to \cite{PP1} for more details. By Lemma $3.4$, when $\textnormal{Re}(P(z-s\tilde{r})) <0$, $\zeta(s,z)$ is well defined, and
$$\zeta(s,z) = \widetilde{\Psi}_0(s,z) \exp\left\{ \sum_{n=1}^\infty \frac{e^{nP(z-s\tilde{r})}}{n}\right\},$$
for $\widetilde{\Psi}_0(s,z)$ that is bi-analytic, say in $|z|<\delta$ and $|s-h|<\epsilon$, by Hartogs' Theorem (Theorem $1.2.5$, \cite{kranz}). Both sides of the above expression are analytic and non-zero when $\textnormal{Re}(P(z-s\tilde{r})) <0$. Hence we can take the logarithmic derivative of $\zeta$ with respect to $s$ to obtain,
$$ \frac{\zeta'(s,z)}{\zeta(s,z)}  = \Psi_0(s,z)  - \frac{P'(-s\tilde{r})e^{z+P(-s\tilde{r})}}{1-e^{z+P(-s\tilde{r})}} $$
for some function $\Psi_0(s,z)$ that is bi-analytic in $|z|<\delta, |s-h|<\epsilon$. This provides the required analytic extension.
\end{proof}

To see that $s=h$ is a simple pole  we need the following result.

\begin{lemma}[Ruelle] \label{ruellelem}
	We have that
	\begin{align*}
	&\frac{d}{ds} P(-s\tilde{r})\Big|_{s=h} = - \int \tilde{r} \ d\mu_{-h\tilde{r}} <0\\
	&\frac{d}{dz} P(z-s\tilde{r})\Big|_{(h,0)} = 1,
	\end{align*}
	where $\mu_{-h\tilde{r}}$ is the 
	equilibrium state for $-h \tilde{r}$.
\end{lemma}
\begin{proof}
	See Proposition $4.10$ of \cite{PP1} and \cite{ruelle-book}.
\end{proof}

We also recall the following standard result which we will use to prove the non-degeneracy of our central limit theorem in Theorem $1.6$.
\begin{lemma}
	The pressure $s \mapsto P(-s\tilde{r})$ is strictly convex at $h$, i.e 
	$$ \frac{d^2}{ds^2} P(-s\tilde{r}) \Big|_{s=h} > 0.$$
\end{lemma}

\begin{proof}
This is well known, see \cite{PP2}.
\end{proof}

 Lemma \ref{ruelle} allows us to apply the implicit function theorem to deduce that $\zeta'(s,z)/\zeta(s,z)$ has a pole near $h$ at $s = \sigma(z)$, implicitly defined by 
$$z + P(-\sigma(z)\tilde{r})= 0 \eqno(5.1),$$
where $s(0) = h$ and $z \mapsto \sigma(z)$ is analytic in a neighbourhood of zero. Furthermore, since
$$\lim_{s \to \sigma(z)} (s-\sigma(z))\frac{P'(-s\tilde{r})e^{z+P(-s\tilde{r})}}{1-e^{z+P(-s\tilde{r})}} = - 1,$$
we see that the logarithmic derivative of $\zeta$ has a simple pole of residue $-1$ at $s=\sigma(z)$. Using the implicit function and Lemma $\ref{ruellelem}$ we obtain the following.

\begin{cor}
	We can write $\sigma'(0) = \frac{1}{\int r d\mu_{-h\tilde{r}}}$.
\end{cor}

We will later see that $\sigma'(0)$ is in fact the constant $A$ from our theorems.

\begin{rem}
	If we let $R$ be the union of the curves corresponding to generators for $\pi_1(V)$
	then $\sigma'(0) = A $ can be interpreted as the total induced measure of $R$ from the Bowen-Margulis measure $\mu$ of maximal entropy, i.e., $A =  \lim_{\epsilon\to 0} \frac{1}{\epsilon}\mu(\phi_{[0, \epsilon]} S_RV)$. 
	In the case of a constant curvature 
	surface of genus $g$ and the generators described in \cite{series} by  \cite{AKU} one  has the geomeric interpretation $A = \hbox{length}(\partial R)/\pi^2(2g-2)$
	where $g \geq 2$ is the genus of the surface.
\end{rem}

We now turn to the case that $|\mathrm{Im}(s)| \ge 1$.   We see from the definition of $\zeta(s,z)$ and Lemma \ref{dolgo} that the following is true.

\begin{lemma}
The function $\zeta(s,z)$ is bi-analytic and non-zero for $|z|<\delta$ and $\mathrm{Re}(s) > h -\epsilon$ and $|\mathrm{Im}(s)| \geq 1$.  In particular, the logarithmic derivative $\frac{\zeta'(s,z)}{\zeta(s,z)}$ 
with respect to $s$, is bi-analytic for $|z|<\delta$ and $\mathrm{Re}(s) > h -\epsilon$ and $|\mathrm{Im}(s)| \geq 1$. 
\end{lemma}

It remains to prove part 2. of Proposition \ref{thmdom2}, which requires a little more work.
We take $h'$ with $h-\epsilon < h' < h$  and consider $\zeta(s,z)$ for $|z| < \delta$ and $\mathrm{Re}(s) > h'$. There exists constants $C_1, C_2 >0$ independent of $z,s$ such that
$$
\begin{aligned}
\left|\log \zeta(s,z) \right| &\leq C_1 \sum_{n=1}^\infty e^{n\mathrm{Re}(z)}  |t|   \rho_0^n
+ kC_2 \sum_{m=0}^\infty e^{m \mathrm{Re}(z)} \rho^{m[\log |t|]}
\left(\sum_{l=0}^{[\log |t|]-1} e^{l |\mathrm{Re}(z)|} e^{lP(-\sigma r)}\right) \cr
&\leq \frac{C_1e^{\mathrm{Re}(z)}\rho_0}{1-\rho_0 e^{\mathrm{Re}(z)}}|t|
+ kC_2 \left(
\frac{1}{1 - e^{|\mathrm{Re}(z)|} |t|^{-|\log \rho|}}
\right) \left( \frac{e^{[\log|t|] |\mathrm{Re}(z)|} e^{[\log |t|] P(- h' r)}-1}{e^{|\mathrm{Re}(z)|}e^{P(- h' r)} -1} \right)
\end{aligned}
$$

 In particular, we can choose $C_3,C_4>0$ independent of $s,z$ such that when $|t|$ is sufficiently large, 
$$
\left|\log \zeta(s,z) \right| \leq C_3 |t| + C_4 |t|^{|\mathrm{Re}(z)|}  |t|^{P(- h' r)}
$$
Moreover, we can choose the exponent of $|t|$ arbitrarily small by reducing $\delta$ and taking $h'$ closer to $h$. We deduce that there exists $\epsilon >0$ such that, 
$$ |\log \zeta(s,z)| = O\left(|\mathrm{Im}(s)|\right)$$
when $\mathrm{Re}(s)>h - \epsilon$, $|z|<\delta$ and furthermore that the implied error constant is independent of $s,z$.

To pass this bound to the logarithmic derivative, we use the method presented in \cite{PS-error} which requires the following result from complex analysis. The main difficultly is ensuring that our bound is uniform in $|z| < \delta$.

\begin{lemma} [\cite{ellison}]\label{inequ}
Let $y \in \mathbb{C}$. Given $R>0$ and $\epsilon>0$ suppose $F$ is analytic on the disk $\Delta = \{s=\sigma + it : |s-y| \le R(1+\epsilon)^3\}$ and that there are no zeros for $F(s)$ on the open subset
$$\{s=\sigma+it \in \mathbb{C} : |s-y| \le R(1+\epsilon)^2 \text{ and } \sigma > \mathrm{Re}(y) - R(1 + \epsilon)\}.$$
Suppose further that there exists a constant $U(y) \ge 0$ such that $\log|F(s)| \le U(y) + \log|F(y)|$ on the set $\{s= \sigma +it: |s-y| \le R(1+\epsilon^3)\}$. Then on the disk $\{s=\sigma+it : |s-y| \le R\}$,
$$ \left|\frac{F'(s)}{F(s)}\right| \le \frac{2+\epsilon}{\epsilon} \left( \left| \frac{F'(y)}{F(y)}\right| + \frac{ \left( 2+ \frac{1}{(1+\epsilon)^2} \right) (1+\epsilon)}{R\epsilon^2} \ U(y) \right).$$
\end{lemma}

To apply this lemma we need to know the location of the zeros for $\zeta$. For our purposes, it suffices to prove the following.

\begin{lemma} \label{axis}
We can find $\epsilon, \delta >0$ such that $\zeta$ is bi-analytic and non-zero on  $|\mathrm{Re}(s) - h| < \epsilon$, $|z|<\delta$ except for simple poles at $s=\sigma(z)$.
\end{lemma}
\begin{proof}
Lemma  5.4  along with our analysis of $\zeta$ for $s$ close to $h$ shows that there exists $\epsilon, \delta, \alpha>0$ for which the required conclusions hold for $\zeta(s,z)$ when $|z|<\delta$ and
$$s \in \{ s \in \mathbb{C}:  |\mathrm{Re}(s) - h| < \epsilon \text{ and } |\mathrm{Im}(s)|< \alpha\} \cup \{s \in \mathbb{C} :  |\mathrm{Re}(s) - h| < \epsilon \text{ and } |\mathrm{Im}(s)| \ge1\}.$$
Hence to conclude the proof we need study the behaviour of $\zeta(s,z)$ when $\alpha \le | \mathrm{Im}(s)| < 1.$ Specifically, we need to show that we can uniformly extend the domain of analyticity for $\zeta(s,z)$ past the line $\mathrm{Re}(s)=h$ for $|z|<\delta$ and $\alpha \le |\mathrm{Im}(s)| \le 1$ and that $\zeta$ is non-zero on this extension. \\
\indent As before, it suffices to work with our symbolic model.  It is a well known consequence of the weak mixing of the geodesic flow that that the spectral radius of $\widetilde{\mathcal L}_s$ is strictly less than $1$ when $s=h + it$ for $t \neq 0$ \cite{PP1}. This property persists under a small perturbation by upper semi-continuity of the spectrum. Using a simple compactness argument we deduce that for $\delta >0$ sufficiently small, the spectral radius of $\widetilde{\mathcal L}_{s,z}$ can be bounded above uniformly by some $0<\rho <1$ when $|z|<\delta$ and $\alpha \le |\mathrm{Im}(s)| \le 1$. As a consequence we can find $\epsilon'>0$ for which $\zeta(s,z)$ is bi-analytic on $|z|<\delta$, $s \in \{s:  |\mathrm{Re}(s) - h| < \epsilon' \text{ and } \alpha \le |\mathrm{Im}(s)| \le 1\}$ and that furthermore, $\zeta$ is non-zero on this domain.  
\end{proof}

We can now use this result to obtain bounds on the logarithmic derivative of $\zeta$. Take $h-\epsilon < h' < h$ and fix $|z|<\delta$. Set $y_z = \sigma(\mathrm{Re}(z)) + 1 + it$ and note that there exists $C>0$ for which
$$|\zeta(y_z,z)| \ge \frac{1}{|\zeta(\mathrm{Re}(y_z), \mathrm{Re}(z))|} \ge C > 0$$
uniformly in both $|t|>1$ and $z$. Hence we may choose $U(y_z)$ in the above lemma for which there exist $\widetilde{C} >0$ satisfying
$$U(y_z) \le \widetilde{C} |t|.$$
Furthermore $\widetilde{C}$ can be taken to be independent of $z,t$. Now, setting $R_z = 1 + \frac{\sigma(\mathrm{Re}(z))-h'}{2}$ we can choose $\epsilon_z>0$  satisfying $R_z(1+\epsilon_z)^3 = 1+(\sigma(\mathrm{Re}(z)) - h')$. Then, as long as $\delta$ is sufficiently small and $h'$ is close enough to $h$, $\Delta$ is contained in the half plane of analyticity for the map $s \mapsto \zeta(s,z)$. Due to Lemma \ref{axis}, we can then apply the Lemma \ref{inequ} to deduce that there exists $C'$ independent of $z$ such that
$$\left|\frac{\zeta'(s,z)}{\zeta(s,z)}\right| \le C' |\mathrm{Im}(s)|$$
when $\mathrm{Re}(s) > \frac{h'+\sigma(\mathrm{Re}(z))}{2}$. To see that $C'$ is independent of $s,z$ satisfying the above conditions note that the following properties hold,
\begin{enumerate}
\item $R_z$ and $\epsilon_z$ are uniformly bounded away from $0$ for all $|z|$ sufficiently small,\
\item for fixed $y$, the right hand side of the inequality in Lemma $3.12$ varies continuously in $\epsilon, R >0$,\
\item there exists $\nu >0$ such that $|\zeta'(y_z,z)/\zeta(y_z,z)|$ is uniformly bounded in $|z| < \nu$.
\end{enumerate}
 By reducing $\delta$ a further time, we may ensure that the above bound holds on a half plane that is independent of $z$, i.e there exists $\delta', \epsilon'>0$ such that for $|z|<\delta'$ and $ \mathrm{Re}(s)>h-\epsilon'$, the above bound holds and the constant is independent of both $z$ and $s$.\\
 
For fixed $z$ we can apply the Phragm\'en Lindel\"of Theorem (\cite{titchmarsh}, $5.65$) to the function 
$$\chi_0(s,z) = \frac{\zeta'(s,z)}{\zeta(s,z)} + \frac{1}{s-\sigma(z)}$$
to deduce that, we can find $\epsilon', \delta'>0$ such that
$$ \left|\chi_0(s,z)\right| \le C |\mathrm{Im}(s)|^{k(\mathrm{Re}(s))}$$
for some $C>0$, when $h-\epsilon' <\mathrm{Re}(s)< h+\delta'$ and $k: [h-\epsilon', h+\delta'] \to \mathbb{R}$ is the linear function that takes value $1$ at $h-\epsilon'$ and value $0$ at $h+\delta'$. Furthermore, as long as $\delta$ is sufficiently small then we can take $\epsilon', \delta'$ such that $C$ is independent of $|z| < \delta$ (i.e. the above bound holds for all $|z| <\delta$, $h - \epsilon' < \mathrm{Re}(s) < h + \delta'$) and also that $h-\epsilon' < \mathrm{Re}(\sigma(z)) < h + \delta'$ for $|z| < \delta$. More specifically, we can take
$$ C = \sup_{|z|<\delta, s \in V} \left|\chi_0(s,z)(-is)^{-k(s)}\right| < \infty$$
where $V = \{s \in \mathbb{R}: h-\epsilon' \le s \le h +\delta'\} \cup \{s: \mathrm{Re}(s)=h-\epsilon'\} \cup \{s: \mathrm{Re}(s)=h+\delta'\} $. It is then clear that, by increasing $C$ by a sufficient amount, the following bound holds uniformly for $h-\epsilon' < \mathrm{Re}(s) < h + \delta'$ with $|\mathrm{Im}(s)|\ge 1$ and $|z|<\delta$,
$$ \left|\frac{\zeta'(s,z)}{\zeta(s,z)}\right| \le C |\mathrm{Im}(s)|^{k(\mathrm{Re}(s))}.$$
Hence, by reducing $\epsilon'$ slightly, we can find $0< \xi < 1$ such that
 $$ \left|\frac{\zeta'(s,z)}{\zeta(s,z)}\right| \le C |\mathrm{Im}(s)|^{\xi}$$
uniformly for all $h-\epsilon' < \mathrm{Re}(s) < h+\delta'$, $|\mathrm{Im}(s)| \ge1$ and $|z|<\delta$. In fact, by reducing $\delta'$ we can take $\xi$ arbitrarily close to $0$. By reducing $\delta$, we may also assume that $h-\epsilon' < \mathrm{Re}(\sigma(z)) < h +\delta'$ for all $|z| < \delta$.

\subsection{Deducing Proposition \ref{thmdom2}}
We are now ready to finish our proof of Proposition \ref{thmdom2}.

\begin{proof} [Proof of Proposition \ref{thmdom2}]
We simply need to combine our work from the previous two subsections. Our study of when $s$ is close to $h$ and Lemma \ref{axis} shows that $\chi_0$ has the required analytic properties. Furthermore, these results show that we can take the logarithmic derivative. The required bounds on this derivative were proved at the end of the previous subsection.
\end{proof}

Computing the logarithmic derivative and using Lemma $2.1$ gives the following.

\begin{lemma} \label{logthm}
Let $\epsilon, \delta$ be as in Proposition \ref{thmdom2}. Then, for $\textnormal{Re}(s) > h -\epsilon$, $|z|<\delta$ we have that
$$ \frac{\zeta'(s,z)}{\zeta(s,z)} = \sum_{\gamma'} l(\gamma') e^{z|\gamma'| -sl(\gamma')} + \chi_1(s,z),$$
where $\chi_1(s,z)$ is analytic and uniformly bounded as $s,z$ vary as above. The sum over $\gamma'$ is over all (not necessarily prime) periodic orbits.
\end{lemma}


\section{Asymptotic results}

We can now convert the analytic properties  of $\zeta$ from Proposition \ref{thmdom2}  into our asymptotic 
results. To do so we use the  Effective Perron Formula to study a certain moment generating function. This  will provide us with the hypotheses for the aforementioned Hwang Quasi-power Theorem. We want to find an asymptotic formula for the following moment generating function.
\begin{definition}
For $z \in \mathbb C$ and $T > 0$ let
$C_z(T) = \sum_{l(\gamma) < T} e^{z|\gamma|}$.
\end{definition}

\subsection{The Perron Formula revisited}
Our aim is to understand the growth of $C_z(T)$ as $T\to\infty$ for all $z$ close to zero.
To do this we want to prove a generalisation of Lemma \ref{qp1} that introduces a dependence on the variable $z$. The main technical difficulty is dealing with additional complex perturbations $e^{nz}$ which can contribute the growth of the sum in the logarithmic derivative of $\zeta$. 
In contrast to the approaches in \cite{vivandval} and \cite{lee}, 
we  study how the real part of the pole $\sigma(z)$ varies for $|z| < \delta$ which will in turn allows us to effectively apply Lemma \ref{perron}. 

\begin{lemma} \label{pressinequ}
We have that
$$\mathrm{Re}(\sigma(z)) \le \sigma(\mathrm{Re}(z)),$$
for all $|z|<\delta$.
\end{lemma}

\begin{proof}
Otherwise, if $z$ does not satisfy this,
$$ \mathrm{Re}(\sigma(z)) > \sigma(\mathrm{Re}(z))$$
and, using $(5.1)$ and the fact that $z \mapsto P(-zr)$ for $z \in \mathbb{R}$ is a decreasing real valued function,
$$-\mathrm{Re}(z) = \mathrm{Re}(P(-\sigma(z)\tilde{r})) \le P(-\mathrm{Re}(\sigma(z))\tilde{r}) <  P(-\sigma(\mathrm{Re}(z))\tilde{r}) = -\mathrm{Re}(z),$$
a contradiction.
\end{proof}

 Using this result along with Lemma \ref{perron}, we deduce the following.

\begin{lemma} \label{qp}
There exists $\alpha >0$ such that
$$\sideset{}{'}\sum_{l(\gamma') \le T}l(\gamma') e^{z|\gamma'|} = \frac{e^{\sigma(z) T}}{\sigma(z)}  \left( 1 + O\left(e^{-\alpha T} \right) \right)$$
uniformly in $z$ and $T$ as $T\to\infty$.
\end{lemma}

\begin{proof}
Set $d = \sigma(\mathrm{Re}(z)) + T^{-1}$, $\epsilon =e^{-xT}$ and $R= e^{yT}$  (for $x,y >0$ to be chosen later). Using Perron's Formula gives that
$$\sideset{}{'}\sum_{l(\gamma') \le T} l(\gamma') e^{z|\gamma'|} = \frac{1}{2\pi i} \int_{d-iR}^{d+iR} \left(- \frac{\zeta'(s,z)}{\zeta(s,z)}\right) \frac{e^{sT}}{s}  ds + \sum_{\gamma'} O\left( \frac{l(\gamma') e^{\mathrm{Re}(z) |\gamma'|} e^{d(T-l(\gamma'))}}{1+R|T-l(\gamma')|}  \right).$$
Now note that, since $\pi(T) = \textnormal{li}(e^{hT}) + O(e^{h'T})$ for some $h' < h$ \cite{PS-error}, if $\alpha \in \mathbb{R}$ satisfies $0<\alpha<h-h'$, then
$$\#\{ \gamma: |l(\gamma) - T| \le e^{-\alpha T}\} = O(e^{(h-\alpha)T}).$$
Hence, if $0<x<h-h'$ and $|T-l(\gamma')|<\epsilon$, then $e^{d(T-l(\gamma'))} = O(1)$ and so
$$\sum_{ |T-l(\gamma')|\le \epsilon}   \frac{ l(\gamma') e^{\mathrm{Re}(z)|\gamma|} e^{d(T-l(\gamma'))}}{1+R|T-l(\gamma')|} = O\left(T e^{\mathrm{Re}(z) c(T+\epsilon)}  \epsilon e^{hT} \right),$$
where we have used that $|\gamma'|<c l(\gamma')$ for some $c>0$. Hence the above is
$$O\left( T \epsilon e^{T(h + \mathrm{Re}(z)c)} \right).$$
By Lemma \ref{logthm},
$$ \sum_{ |T-l(\gamma')|> \epsilon}   \frac{ l(\gamma') e^{\mathrm{Re}(z)|\gamma'|} e^{d(T-l(\gamma'))}}{1+R|T-l(\gamma')|} \le \frac{e^{dT}}{R\epsilon} \sum_{\gamma'} l(\gamma')  e^{\mathrm{Re}(z)|\gamma'|}e^{-dl(\gamma')} = \frac{e^{dT}}{R\epsilon} \left( \frac{\zeta'(d,\mathrm{Re}(z))}{\zeta(d,\mathrm{Re}(z))} +O(1)\right).$$
Then, using Proposition $5.2$, $\frac{\zeta'(d,\mathrm{Re}(z))}{\zeta(d,\mathrm{Re}(z))} = O(T)$ and so the above is
$$O\left( \frac{ T e^{dT}}{R\epsilon}\right).$$

Now fix $c<h$ such that $\chi_0(s,z)$ (from Proposition $5.2$) is bi-analytic in $\mathrm{Re}(s) > c$, $|z|<\delta$. By reducing $\delta$ if necessary we may assume that $\mathrm{Re}(\sigma(z)) > c$ for $|z|<\delta$. We integrate around the rectangle $\Gamma$ with vertices $d-iR, c-iR, c+iR,d+iR$. By the Residue Theorem (by Lemma \ref{pressinequ} we pick up the residue)
$$\frac{1}{2\pi i} \int_\Gamma \left(- \frac{\zeta'(s,z)}{\zeta(s,z)}\right) \frac{e^{sT}}{s} ds = \frac{e^{\sigma(z)T}}{\sigma(z)}.$$
Our bounds on  $\eta$ guarantee that
$$\int_{c \pm iR}^{d \pm iR} \left(- \frac{\zeta'(s,z)}{\zeta(s,z)}\right) \frac{e^{sT}}{s} ds = O(R^{\xi-1} e^{dT})$$
and
$$\int_{c-iR}^{c+iR} \left(- \frac{\zeta'(s,z)}{\zeta(s,z)}\right) \frac{e^{sT}}{s} ds = O(R^\xi e^{cT}).$$
All of the above big O bounds can be chosen to be uniform in $|z|<\delta$. We now choose $y$ such that $R^\xi e^{cT}$ grows strictly exponentially slower that $e^{\sigma(z)T}$ uniformly for $|z| < \delta$. Now note that the continuous function $z \mapsto\mathrm{Re}(\sigma(z)) - \sigma(\mathrm{Re}(z))$ fixes $0$. Hence, by reducing $\delta$ we can assume that $R^{\xi-1}e^{dT}$ grows strictly exponentially slower than $e^{\sigma(z)T}$ uniformly for $|z| < \delta$ too. We then  choose $0 < x < h-h'$ so that $ T e^{dT}/R\epsilon$ grows strictly exponentially slower than $e^{\sigma(z)T}$ again uniformly in $|z| < \delta$. Reducing $\delta$ (if necessary) a final time guarantees that $T \epsilon e^{T(h + \mathrm{Re}(z)c)}$ grows in the required way also.
\end{proof}

A standard argument 
again shows the same is true if we sum over prime periodic orbits and  it is easy to remove the terms corresponding to $l(\gamma) = T$ from our asymptotic expression. We deduce that there exists  some $\alpha>0$ for which

\begin{equation*}
 \sum_{l(\gamma ) < T}    l(\gamma) e^{z|\gamma|} =  \frac{ e^{\sigma(z) T}}{\sigma(z)} \left( 1 + O\left( e^{-\alpha T} \right) \right).
\end{equation*}
We now need to remove the $ l(\gamma)$ term from the above expression.

\begin{lemma}
We  have that
$$ C_z(T) =  \frac{ e^{\sigma(z) T}}{\sigma(z) T} \left( 1 + O\left( \frac{1}{T} \right) \right),$$
where the implied error term is independent of $z$.
\end{lemma}

\begin{proof}
Let 
$$\varphi_z(u) = \sum_{l(\gamma) < u} l(\gamma) e^{z|\gamma|}.$$
Then, 
$$C_z(T) = \int_0^T \frac{1}{u} \ d\varphi_z(u) = \frac{\varphi_z(T)}{T} + \int_0^T \frac{\varphi_z(u)}{u^2} \ du.$$
From our asymptotic expression for $\varphi_z$, we have that
$$\left| \int_0^T \frac{\varphi_z(u)}{u^2} \ du \right| \le \int_0^T \frac{|\varphi_z(u)|}{u^2} \ du = \int_0^T \frac{e^{\mathrm{Re}(\sigma(z))u}}{|\sigma(z)|u^2} \ du + O(e^{(\sigma(z)-\alpha')T}),$$
for some $\alpha' >0$. The conclusion then follows from a simple calculation.
\end{proof}

\subsection{Hwang's Quasi-Power Theorem}
We can now deduce Theorems $1.4$ and $1.5$ using the the Hwang Quasi-Power Theorem. This result allows us 
to  convert our estimate for $C_z(T)$ into estimates on the average, variance and a Central Limit Theorem for the distributions on the positive real numbers associated to
$$
X_T:=  \frac{1}{\pi(T)}\sum_{l(\gamma) < T} \delta_{|\gamma|}.
$$
An excellent account of this theory appears in (\cite{fs}, \S IX.5). 
\begin{lemma}[Hwang's Theorem \cite{hwang1}, \cite{hwang2}]
Assume that 
	$$
\frac{C_{z}(T)}{C_0(T)}   = \exp \left( \beta_T U(z) + V(z)\right) \left( 1 + O(1/\kappa_T) \right)
	$$
	with $\beta_T, \kappa_T \to \infty$ as $T \to \infty$, 
	where 
	$U$ and $V$ are analytic in a neighbourhood of $z=0$.  Then  
	$$
	\mathbb E \left(X_T\right) = \beta_T U'(1) + V'(1) +   O(1/\kappa_T)
	\hbox{ and } 
	\mathbb V \left(X_T\right) = \beta_T U''(1) + V''(1) +   O(1/\kappa_T)
	$$
	and $X_T$ is asymptotically Gaussian with error term $O(\kappa_T^{-1}+ \beta_T^{-1/2})$.
\end{lemma}

From our work above we have the expression
$$ \frac{C_z(T)}{C_0(T)} = \frac{h}{\sigma(z)} e^{(\sigma(z)-h)T}\left( 1 + O\left(\frac{1}{T}\right) \right)$$
and can therefore apply Hwang's Theorem with
$$\beta_T = T, \ \ \ \kappa_T =T, \ \ \ \ \ \ U(z) = \sigma(z) - h  \ \ \text{ and } \ \ v(z) = \log\left( \frac{h}{\sigma(z)}\right)$$
to  deduce Theorem $1.4$ and Theorem $1.5$. 
It only remains to compute the constants associated to those results.


\subsection{Computing constants}
We begin by computing the constants associated to $U(z)$. Recall that we used the implicit function theorem to obtain our analytic function $\sigma : \{z:|z|<\delta\} \to\mathbb{C}$ such that $ z + P(-\sigma(z)\tilde{r}) = 0$ for $|z|<\delta$. Differentiating this expression, we obtain
$$
-P'(-\sigma(z) r)\sigma'(z) = 1
\eqno(6.2)
$$
and so
$$A: = \sigma'(0) = \frac{1}{ \int r\ d\mu_{-h\tilde{r}}}$$
where, as before, $\mu_{-h\tilde{r}}$ is the equilibrium state from $-h\tilde{r}$. Differentiating expression (6.2) gives
$$ \sigma''(z) = \frac{P''(-\sigma(z)\tilde{r})\sigma'(z)}{(P'(-\sigma(z)\tilde{r}))^2}$$
and so
$$\sigma^2 := \sigma''(0) = -\frac{P''(-h\tilde{r})}{(P'(-h\tilde{r}))^3} = \frac{P''(-h\tilde{r})}{(\int r \ d\mu_{-h\tilde{r}})^3} >0.$$

The other function $v(z)$ is given by $v(z) = \log(h) -\log(\sigma(z))$. Hence
$$v'(z) = - \frac{\sigma'(z)}{\sigma(z)} \ \ \text{ and } \ \  v'(0) = - \frac{1}{h \int r \ d\mu_{-h\tilde{r}}}.$$
We then have that
$$ v''(z)= - \frac{\sigma''(z)\sigma(z) - (\sigma'(z))^2}{(\sigma(z))^2}$$
and
$$D:= v''(0) = -\frac{1}{h^2 (\int r \ d\mu_{-h\tilde{r}})^2} \left( \frac{hP''(-h\tilde{r})}{\int r \ d\mu_{-h\tilde{r}}} - 1 \right).$$
Letting $P''(-h\tilde{r}) = P$ gives the following relations,
\begin{enumerate}
\item  $\sigma^2 = P A^3$,
\item $D = (A/h)^2 - \sigma^2/h.$
\end{enumerate}


\begin{rem}
Rather than expressing these values in terms of thermodynamic quantities, 
it is a straightforward exercise to rewrite them purely in terms of the zeta functions and their derivatives.
\end{rem}

\section{Ordering orbits by word length}
In this section we consider the analogous asymptotic results when we order periodic orbits by their word length instead of geometric length. To do so we, as before, study an appropriate moment generating function with the aim of applying Hwang's Theorem. However, in  this setting  our proofs are much simpler as we do not need to use zeta functions. Throughout this section we work in the symbolic model for our interval map $T: I \to I$, i.e. we use Bowen-Series coding \cite{bowenseries}. 
In this setting the appropriate moment generating function to study is defined as follows.
\begin{definition}
Let
$$E_z(T) = \sum_{|\gamma| = T} e^{z l(\gamma)}$$
where the sum is taken over prime orbits, for $z \in \mathbb{C}$ and $T \in \mathbb{Z}_{\ge 0}$.
\end{definition}
As in our previous work, it suffices to study $E_z$ when the sum is taken over all, not necessarily prime periodic orbits. From now on we will use this different definition. We will abuse notation and call this function $E_z$ also. Using our symbolic model we can write
$$E_z(T) = \sum_{\sigma^T x=x} e^{zr^T(x)}$$
and so, using ideas from the proof of Lemma $3.4$, there exists $\epsilon >0$ such that
$$E_z(T) = e^{TP(zr)}\left( 1 + O(e^{-\delta T})\right)$$
uniformly for $|z| < \epsilon$.
Hence
$$\frac{E_z(T)}{E_0(T)} = e^{T(P(zr)-h)}\left(1 + O(e^{-\beta T})\right)$$
and we can apply Hwang's Theorem with
$$\beta_T = T, \ \ \kappa = \beta T, \ \ U(z)= P(zr)-h \ \ \text{and} \ \ V(z) =0.$$ 
It is well known that $E'(0) = \int r \ d\mu =: \widetilde{A}$ where $\mu$ is the measure of maximal entropy for our subshift. Furthermore, 
$$E''(0) = \lim_{n\to\infty} \frac{1}{n} \int \left(r^n-n\int r \ d\mu\right)^2 d\mu >0.$$
Writing $\mathfrak{C}_T$ for the number of prime closed geodesics of word length $T$, we deduce the following.
\begin{prop}
There exists $\beta >0$ such that
$$\frac{1}{\mathfrak{C}_T} \sum_{|\gamma|=T} l(\gamma)= \widetilde{A} T + O(e^{-\beta T})$$
as $T\to\infty$. Furthermore, there exists $\sigma^2>0$ such that
$$ \frac{1}{\mathfrak{C}_T} \#\left\{ |\gamma| =T: \frac{l(\gamma) - \widetilde{A}T}{\sqrt{T}} \le x \right\}$$
converges to a normal distribution with mean $0$ and variance $\sigma^2$ determined by
$$ \frac{1}{\mathfrak{C}_T} \sum_{|\gamma|=T} \left( l(\gamma)^2 - \frac{1}{\mathfrak{C}_T} \sum_{|\gamma|=T} l(\gamma)\right)^2 = \sigma^2T + O(e^{-\beta T}). $$
Furthermore, the convergence rate for this central limit theorem is of Berry-Esseen type.
\end{prop}

As before, we can find an asymptotic expansion for the asymptotic average between $|\gamma|$ and $l(\gamma)$. This follows from a simple computation and so we exclude the proof.
\begin{cor}
For any $k\in\mathbb{Z}_{\ge 0}$ we can write
$$\frac{1}{\mathfrak{C}_T} \sum_{|\gamma|=T} \frac{l(\gamma)}{|\gamma|} = \sum_{j=0}^{k-1} \frac{a_j}{T^j} + O\left( \frac{1}{T^k}\right).$$
Furthermore, $a_0 = \widetilde{A}, a_2= \widetilde{A}/h$ and $a_3= \widetilde{A}/h^2$. 
\end{cor}

Our method here provides a alternate proof of Gekhtman, Taylor and Tiozzo's central limit theorem \cite{gtt} in our setting. Moreover, it shows that the Berry-Esseen error term holds. Moreover, it is easy to see that our method works at the full level of generality considered by Gekhtman, Taylor and Tiozzo.

 \section{Local Limit Theorems}
 
 In this section we prove our local limit theorem, Theorem $1.7$. As before, we want to understand an appropriate quasi-power expression. Specifically we want to study the behaviour of
 $$\frac{C_{it}(T)}{C_0(T)} = \frac{1}{\pi(T)} \sum_{l(\gamma) <T} e^{it|\gamma|}$$
 for $t \in \mathbb{R}$ as $T\to\infty$. We want to understand how $C_{it}(T)$ grows as $t$ varies in $ \mathbb{R}$ (or $[-\pi,\pi]$)  instead of belonging to a small complex neighbourhood of zero. As before, to achieve this, we study our zeta function of two variables. The fact that $it$ is imaginary makes our analysis easier as the perturbing factor $e^{it}$ has unit modulus and so  does not contribute to the growth of the quantities we want to consider. We have the following.
 
 \begin{prop}
 For ever $v>0$, there exists $\xi >0$ such that,
 \begin{enumerate}
 \item for any fixed $|t| \in [v, \pi]$, $s \mapsto \zeta'(s,it)/\zeta(s,it)$ is well defined and analytic for $|\mathrm{Re}(s)-h|<\xi$; and
 \item there is $M>0$ such that, 
 $$\left|\frac{\zeta'(s,it)}{\zeta(s,it)}\right| \le M \max\{|\mathrm{Im}(s)|^{\beta}, 1\}$$
for some $0<\beta <1$ uniformly in $|\mathrm{Re}(s)-h| <\xi$, $|t| \in [v, \pi]$.
 \end{enumerate}
 \end{prop}

\begin{proof}
As before we analytically extend $\zeta$ locally to $s=h$ to obtain
$$\frac{\zeta'(s,it)}{\zeta(s,it)} = \frac{e^{it}e^{P(-sr)}}{1-e^{it}e^{P(-sr)}} + \widetilde{\chi}_0(s,it)$$
where, for any fixed $t\in\mathbb{R}$, $s \mapsto \widetilde{\chi}_0(s,it)$ is analytic in $|s-h|<\epsilon$ for some $\epsilon >0$. Again, we see that there is a simple pole at $s,it$ when $it + P(-sr)=0$. Then note that we can find $c,\delta>0$ such that $|1- e^{it}e^{P(-sr)})| \ge c$ for all $|t| \in [v,\pi]$ and $|s-h|<\delta$. We conclude that there exists $\epsilon >0$ such that for any fixed $|t| \in [v,\pi]$, the map $s \mapsto \zeta'(s,it)/\zeta(s,it)$ is well defined and analytic for $s$ in the domain $|s-h|<\epsilon$. The analysis used to study $\zeta(s,z)$ when $z$ was in a complex neighbourhood of zero can be applied again and allows us to extend this domain to a half plane as required. \\
\indent Part $2.$ of the proposition follows from the same method used to obtain the logarithmic derivative bounds in Proposition $\ref{thmdom2}$. As mentioned previously, the fact that $e^{it}$ has unit modulus makes this easy to verify. 
\end{proof}
 
 Applying the same arguments used in Lemma $\ref{qp}$ and the paragraphs proceeding its proof gives the following. Note that since $\zeta'(s,it)/\zeta(s,it)$ is analytic on the domain described in the previous proposition, there is no residue to pick up when we apply the Residue Theorem.
 
 \begin{lemma} \label{imqp}
 Take $v>0$, then there is some $\beta >0$ such that
 $$\frac{C_{it}(T)}{C_0(T)} = \frac{1}{\pi(T)} \sum_{l(\gamma) <T} e^{it|\gamma|} = O\left(e^{-\beta T}\right)$$
 uniformly for $|t| \in [v,\pi]$.
 \end{lemma}
 \begin{proof}
 We use the same method as Lemma $\ref{qp}$. Setting $d = h + T^{-1}$, $\epsilon = e^{-xT}$ and $R=e^{yT}$ for appropriate $x,y \in \mathbb{R}_{\ge 0}$ and using Perron's Formula along with the same estimates gives
 $$ \sideset{}{'}\sum_{l(\gamma') <T} l(\gamma') e^{it|\gamma'|} = O\left(e^{(h-\delta) T}\right)$$
 for some $\delta>0$. We can then run the same arguments and use the Stieltjes integral to obtain the desired expression.
 \end{proof}
 
 We now follow the method presented in \cite{vivandval}. To simplify the notation throughout the following, let 
 $$\widetilde{C}_{it}(T) = \frac{C_{it}(T)}{C_0(T)}$$
 and write, for $l\in\mathbb{Z}_{\ge0}$,
 $$P_T(l) = \frac{1}{\pi(T)} \#\{\gamma: l(\gamma) < T, |\gamma|=l\}.$$
 Theorem $1.7$ is concerned with the asymptotic growth of 
 $P_T([AT + \sigma x \sqrt{T}])$
 for $x \in\mathbb{R}$, where $[\cdot]$ denotes the nearest integer function. Let $l_x(n) = AT + \sigma x \sqrt{T}$ and define
 $I_T(x) = 2\pi\sqrt{T} \ P_T([l_x(T)]).$
 Since $\widetilde{C}_{it}(T) = \sum_{k \ge 0} P_T(k) e^{itk}$, we have that
 $$I_T(x) = \sqrt{T} \int_{-\pi}^\pi e^{-it[l_x(T)]} \widetilde{C}_{it}(T) dt.$$
 Now recall that, by our previous work, when $v>0$ is sufficiently small, $\widetilde{C}_{it}(T)$ admits the following quasi-power expression for $|t|<v$,
 \begin{equation}\label{qpexpression}
 \widetilde{C}_{it}(T) = \frac{h}{\sigma(it)} e^{(\sigma(it)-h)T}\left( 1+O\left(\frac{1}{T}\right)\right). \tag{8.1}
 \end{equation}
This expression can be used to study
$$I_T^0 = \sqrt{T} \int_{-v}^v  e^{-it[l_x(T)]} \widetilde{C}_{it}(T) dt$$
for small $v$. The following proposition shows that $I_T^0$ provides the lead term for our limit theorem. 
\begin{prop}\label{balval}
We have that
 $$I_T^0(x) = \frac{\sqrt{2\pi}}{\sigma} e^{-x^2/2} + O(T^{-1/2})$$
 as $T\to\infty$.
\end{prop}
\begin{proof}
The proof of this proposition follows the method used by Baladi and Vall\'ee in \cite{vivandval} used to proof Theorem $4$ in Section $5$. The proof relies on our quasi-power expression for $\widetilde{C}_{it}(T)$, (\ref{qpexpression}), and up to some minor changes the proof presented in \cite{vivandval} works in our setting. We therefore refer the reader to \cite{vivandval} for the proof.
\end{proof}

 We are now ready to prove our local limit theorem.
 
 \begin{proof} [Proof of Theorem $1.7$]
By Lemma \ref{imqp},
 $$I_T^1(x) := \sqrt{T} \int_{|t|\in [v,\pi]} e^{-it[l_x(T)]} \widetilde{C}_{it}(T) dt = O\left(\sqrt{T} \sup_{|t| \in [v,\pi]}|\widetilde{C}_{it}(T)|\right) = O(T^{-1/2}).$$
 Hence, by Proposition \ref{balval},
 $$I_T(x) = I_T^0(x) + I_T^1(x) =  \frac{\sqrt{2\pi}}{\sigma} e^{-x^2/2} + O(T^{-1/2}).$$
 Through simple algebraic manipulations, this expression can rearranged into the statement of Theorem $1.7$. This concludes the proof.
 \end{proof}

\Addresses

\end{document}